\documentclass[12pt,reqno,twoside]{amsart}

\usepackage{amssymb,amsmath,amscd,enumerate,verbatim}
\usepackage[pagebackref=true,colorlinks=true,linkcolor=blue,citecolor=blue]{hyperref}

\usepackage[latin1]{inputenc}
\usepackage{color}
\usepackage{graphicx}
\usepackage{tikz}

\usepackage[all]{xy}
\usepackage{pstricks, pst-3d}

\newcommand{\mm}{\mathfrak m}

\newcommand{\kk}{\mathrm k}

\DeclareMathOperator{\pnt}{\raise 0.5mm \hbox{\large\bf.}}

\DeclareMathOperator{\diam}{diam}

\DeclareMathOperator{\Ann}{Ann}

\DeclareMathOperator{\dist}{dist}

\DeclareMathOperator{\supp}{supp}

\let\phi=\varphi
\let\:=\colon

\newtheorem{thm}{\bf Theorem}[section]
\newtheorem{lem}[thm]{\bf Lemma}
\newtheorem{cor}[thm]{\bf Corollary}
\newtheorem{prop}[thm]{\bf Proposition}
\newtheorem{conj}[thm]{\bf Conjecture}

\theoremstyle{definition}
\newtheorem{defn}[thm]{\bf Definition}

\newtheorem{prob}[thm]{\bf Problem}

\theoremstyle{plain}
\newtheorem*{thm*}{Theorem}
\newtheorem*{lem*}{Lemma}
\newtheorem*{cor*}{Corollary}
\newtheorem*{claim*}{Claim}
\newtheorem*{defn*}{Definition}

\theoremstyle{remark}
\newtheorem{rem}[thm]{Remark}

\newtheorem{exm}[thm]{Example}

\numberwithin{equation}{section}
\title{Associated primes of the second power of closed neighborhood ideals of graphs}

\author{Ha Thi Thu Hien}
\address{Foreign Trade University, 91 Chua Lang, Hanoi, Vietnam}
\email{thuhienha504@gmail.com}
\author{Thanh Vu}
\address{Institute of Mathematics, VAST, 18 Hoang Quoc Viet, Hanoi, Vietnam}
\email{vuqthanh@gmail.com}
\thanks{}
\date{\today}
\subjclass[2020]{13C15, 13F55}
\keywords{Associated prime, closed neighborhood ideal,  powers of ideal}

\begin{document}

\begin{abstract} We study simple graphs for which the maximal homogeneous ideal is an associated prime of the second power of their closed neighborhood ideals. In particular, we show that such graphs must have diameter at most $6$, and that those with diameter $2$ must be vertex diameter-$2$-critical.
\end{abstract}

\maketitle

\section{Introduction} 
Lov\'asz introduced the neighborhood complex of a simple graph to prove Kneser's conjecture \cite{L}. Matsushita and Wakatsuki \cite{MW} proved that the Alexander dual of the neighborhood complex is the dominance complex of the complement of $G$. The neighborhood complexes and dominance complexes of graphs are fundamental objects in combinatorics and have been studied for a long time \cite{B}. Nonetheless, the squarefree monomial ideals associated with these complexes were only studied fairly recently. In particular, Sharifan and Moradi \cite{SM} and Honeycutt and Sather-Wagstaff \cite{HS} described the primary decomposition of the closed neighborhood ideal of a simple graph. They studied fundamental algebraic properties of this ideal, namely the unmixed property and the Castelnuovo-Mumford regularity. 

In this work, we study the associated primes of powers of the closed neighborhood ideals of simple graphs. Let us first introduce some notation. Let $G$ be a finite simple graph on the vertex set $V(G) = [n] = \{1,\ldots,n\}$ and the edge set $E(G) \subseteq V(G) \times V(G)$. For a vertex $i \in V(G)$, the neighborhood of $i$ is the set $N_G(i) = \{j \mid \{i,j\} \text{ is an edge in } G\}$. Its closed neighborhood is $N_G[i] = N_G(i) \cup \{i\}$. We denote by $S = \kk[x_1,\ldots,x_n]$ a standard graded polynomial ring over a field $\kk$. For each subset $W \subseteq [n]$, we denote by $x_W$ the monomial $\prod_{j\in W} x_j$. The closed neighborhood ideal of $G$ is defined by
$$NI(G) = (x_{N_G[i]} \mid i = 1, \ldots, n).$$
In particular, $NI(G)$ is generated by at most $n$ squarefree monomials.

In general, there are several combinatorial descriptions of the associated primes of powers of squarefree monomial ideals; see, for instance, \cite{BFT, FHV2, HM, TT}. Nonetheless, these general descriptions are often very difficult to apply in identifying embedded primes of specific powers of a given squarefree monomial ideal. Francisco, Ha, and Van Tuyl \cite{FHV1} showed that the embedded associated primes of the second power of the cover ideal of a graph are precisely those arising from the odd holes of the graph. Lam and Trung \cite{LT} described the associated primes of powers of the edge ideal of a graph in terms of ear decompositions. In other words, when monomial ideals arise naturally from a graph, one often expects a combinatorial description of the associated primes of their powers. Let $\mm = (x_1, \ldots, x_n)$ denote the maximal homogeneous ideal of $S$. Guided by this philosophy, we are interested in the following problem:

\begin{prob}\label{prob_a} Classify graphs $G$ for which $\mm$ is an associated prime of a certain power of $NI(G)$.    
\end{prob}

Nasernejad and Qureshi \cite{NQ} recently proved that when $G$ is a strongly chordal graph then $NI(G)$ is normally torsion-free. In particular, $\mm$ is not an associated prime of $S/NI(G)^t$ for all $t \ge 1$. The problem of classifying graphs for which their closed neighborhood ideals are normally torsion-free was also studied in \cite{NQBM} and is of interest in its own right. Problem \ref{prob_a} is on the other extreme.

As a first step, in this work, we study graphs $G$ for which $\mm$ is an associated prime of the second power of $NI(G)$. By \cite[Theorem 1.1]{NV}, such a graph must be connected. We first show that when $\diam(G) \leq 2$, this condition characterizes exactly the class of vertex diameter-$2$-critical graphs. Before stating the result, we recall the notion of vertex diameter-critical graphs. For a vertex $x \in V(G)$, we denote by $G - x$ the subgraph of $G$ obtained by removing $x$ and all edges incident to $x$.

\begin{defn}
A simple connected graph $G$ with $\diam(G) = d$ is said to be \emph{vertex diameter-$d$-critical} if $\diam(G - x) \geq d + 1$ for every $x \in V(G)$.
\end{defn}

\begin{thm}\label{thm_diam_2}
Let $G$ be a simple graph of diameter at most $2$, and let $NI(G)$ denote the closed neighborhood ideal of $G$. Then $\mm$ is an associated prime of $S/NI(G)^2$ if and only if $G$ is vertex diameter-$2$-critical.
\end{thm}

When $\diam(G) \geq 3$, we show that the condition for $\mm$ to be an associated prime of $S/NI(G)^2$ can be detected by examining a hypergraph constructed from pairs of vertices at distance at least $3$ in $G$. More precisely, we define a hypergraph $H$ with vertex set $V(H) = V(G) = [n]$, and edges given by sets of the form $N_G[i] \cup N_G[j]$ for all pairs $i, j \in V(G)$ such that $\dist_G(i, j) \geq 3$. We then have the following characterization:

\begin{thm}\label{thm_diam_3}
Let $G$ be a simple graph with $\diam(G) \geq 3$, and let $NI(G)$ be its closed neighborhood ideal. Let $H$ be the hypergraph associated to $G$ as described above. Then $\mm$ is an associated prime of $S/NI(G)^2$ if and only if there exists a subset $C \subseteq V(G)$ such that:
\begin{enumerate}
    \item $C$ is a vertex cover of $H$;
    \item for every $i \in C$, there exists an edge $e \in E(H)$ such that $C \cap e = \{i\}$;
    \item for every $j \in V(G) \setminus C$, there exist vertices $j_1, j_2 \in N_G(j)$ such that $N_G[j_1], N_G[j_2] \subseteq V(G) \setminus C$, and 
    \[
    \dist_G(j_1, j_2) = 2, \quad \text{and} \quad \dist_{G - j}(j_1, j_2) \geq 3.
    \]
\end{enumerate}
\end{thm}

In the next section, we recall necessary notation and prove our main results. We then provide examples and applications of our results.

\section{Second powers of closed neighborhood ideals}
Throughout this paper, we denote by $S = \kk[x_1, \ldots, x_n]$ a polynomial ring over a field $\kk$, and by $\mm = (x_1, \ldots, x_n)$ its maximal homogeneous ideal. For a nonzero monomial $f \in S$ and an index $i$ with $1 \leq i \leq n$, we denote by $\deg_i(f)$ the largest exponent $r$ such that $x_i^r$ divides $f$. The \emph{support} of $f$, denoted by $\supp(f)$, is the set of all variables of $S$ that divide $f$.

\begin{defn}
Let $I$ be a nonzero monomial ideal of $S$. For each $i = 1, \ldots, n$, the \emph{$i$-degree} of $I$ is defined by
\[
\rho_i(I) = \max \{ \deg_i(f) \mid f \text{ is a minimal monomial generator of } I \}.
\]
\end{defn}

\begin{defn}
Let $R$ be a Noetherian ring and $M$ a finitely generated $R$-module. A prime ideal $P \subseteq R$ is called an \emph{associated prime} of $M$ if there exists an element $x \in M$ such that
\[
P = \Ann(x) = \{ a \in R \mid ax = 0 \}.
\]
\end{defn}

We begin with a simple lemma.

\begin{lem}\label{lem_quotient}
Let $I$ be a monomial ideal and let $P = (x_1, \ldots, x_t)$ be a monomial prime ideal. Then $P$ is an associated prime of $S/I$ if and only if there exists a monomial $f$ such that $\deg_i(f) < \rho_i(I)$ for all $i = 1, \ldots, t$ and $I : f = P$. 

In particular, if $J$ is a squarefree monomial ideal and $P$ is an associated prime of $S/J^t$, then $\deg_i(f) < t$ for all $i = 1, \ldots, t$.
\end{lem}

\begin{proof}
By \cite[Corollary 1.3.10]{HH}, it suffices to show that if $I : f = P$, then $\deg_i(f) < \rho_i(I)$ for all $i = 1, \ldots, t$. Suppose, for contradiction, that $\deg_1(f) \geq \rho_1(I)$. Since $I : f = P$, there exists a minimal generator $g$ of $I$ such that $g / \gcd(g, f) = x_1$. In particular, $\deg_1(g) > \deg_1(\gcd(g, f))$. 

Because $g$ is a minimal generator of $I$, we have
\[
\rho_1(I) \geq \deg_1(g) > \deg_1(\gcd(g, f)) = \min(\deg_1(g), \deg_1(f)) = \deg_1(f),
\]
which contradicts the assumption that $\deg_1(f) \geq \rho_1(I)$. The result follows.
\end{proof}

\begin{defn}
Let $G$ be a connected simple graph. For vertices $u$ and $v$ in $G$, the \emph{distance} between $u$ and $v$, denoted by $\dist_G(u, v)$, is the length of the shortest path in $G$ connecting them. The \emph{diameter} of $G$, denoted by $\diam(G)$, is the maximum distance between any two vertices in $G$.
\end{defn}

We now prove Theorem~\ref{thm_diam_2}.

\begin{proof}[Proof of Theorem~\ref{thm_diam_2}]
First, assume that $\mm$ is an associated prime of $S/NI(G)^2$. By Lemma~\ref{lem_quotient}, there exists a squarefree monomial $f$ such that $NI(G)^2 : f = \mm$. We claim that $f = x_1 \cdots x_n$. Assume, for contradiction, that $\supp(f)$ is a proper subset of $[n]$. Without loss of generality, assume that $1 \notin \supp(f)$. Since $x_1 f \in NI(G)^2$, there exist vertices $j, k \in V(G)$ such that $x_{N_G[j]} x_{N_G[k]}$ divides $x_1 f$. Since $1 \notin \supp(f)$, $x_1f$ is squarefree. Hence, $N_G[j] \cap N_G[k] = \emptyset$, implying that $\dist_G(j,k) \geq 3$, a contradiction to the assumption that $\diam(G) = 2$.

Now, assume that $G$ is a vertex diameter-$2$-critical graph. Let $f = x_1 \cdots x_n$. We will show that $NI(G)^2 : f = \mm$. Since $\diam(G) = 2$, we have $\dist_G(i,j) \leq 2$ for all $i, j \in V(G)$. Hence, $N_G[i] \cap N_G[j] \neq \emptyset$ for all $i, j$, so $f \notin NI(G)^2$.

Next, for each $i = 1, \ldots, n$, we will show that $x_i f \in NI(G)^2$. Since $G$ is vertex diameter-$2$-critical, we have $\diam(G - i) \geq 3$. Therefore, there exist vertices $j, k \in V(G) \setminus \{i\}$ such that $\dist_{G - i}(j, k) \geq 3$. This implies that $N_G[j] \cap N_G[k] = \{i\}$, and hence $x_{N_G[j]} x_{N_G[k]}$ divides $x_i f$. Therefore, $x_i f \in NI(G)^2$ for all $i = 1, \ldots, n$, and so $NI(G)^2 : f = \mm$, as desired. This completes the proof.
\end{proof}

We now prove Theorem~\ref{thm_diam_3}.

\begin{proof}[Proof of Theorem~\ref{thm_diam_3}]
First, assume that $\mm$ is an associated prime of $S/NI(G)^2$. By Lemma~\ref{lem_quotient}, there exists a squarefree monomial $f$ such that $NI(G)^2 : f = \mm$. Since $\diam(G) \geq 3$, there exist vertices $j$ and $k$ such that $N_G[j] \cap N_G[k] = \emptyset$. Hence, $\supp(f)$ is a proper subset of $[n]$. Let $C = [n] \setminus \supp(f)$. We will show that $C$ satisfies the required conditions.

Let $j, k \in V(G)$ with $\dist_G(j, k) \geq 3$. Then $f$ is not divisible by $x_{N_G[j]} x_{N_G[k]}$, so
\[
N_G[j] \cup N_G[k] \nsubseteq \supp(f) \quad \Longrightarrow \quad (N_G[j] \cup N_G[k]) \cap C \neq \emptyset.
\]
This shows that $C$ is a vertex cover of $H$.

Now let $i \in C$. Since $x_i f \in NI(G)^2$, there exist vertices $j, k \in V(G)$ such that $x_{N_G[j]} x_{N_G[k]}$ divides $x_i f$. That is, $N_G[j] \cup N_G[k] \subseteq \supp(f) \cup \{i\}$. But $x_{N_G[j]} x_{N_G[k]} \nmid f$, so at least one variable in $x_{N_G[j]} x_{N_G[k]}$ must be $x_i$. Hence, $(N_G[j] \cup N_G[k]) \cap C = \{i\}$.

Next, let $j \in V(G) \setminus C$. Since $x_j f \in NI(G)^2$, there exist vertices $j_1, j_2 \in V(G)$ such that
\[
x_{N_G[j_1]} x_{N_G[j_2]} \mid x_j f \quad \text{but} \quad x_{N_G[j_1]} x_{N_G[j_2]} \nmid f.
\]
This implies that $j \in N_G[j_1] \cup N_G[j_2]$, but all other variables in that product belong to $\supp(f)$. Therefore,
\[
N_G[j_1], N_G[j_2] \subseteq \supp(f) = V(G) \setminus C, \quad \text{and} \quad N_G[j_1] \cap N_G[j_2] = \{j\}.
\]
Thus, $\dist_G(j_1, j_2) = 2$ and $\dist_{G - j}(j_1, j_2) \geq 3$.

\medskip

Conversely, assume that there exists a subset $C \subseteq V(G)$ satisfying conditions (i)-(iii). Since $H$ is nonempty, $C$ is nonempty. Let $f = x_{V(G) \setminus C}.$ We will show that $NI(G)^2 : f = \mm$. 

\begin{itemize}
    \item By condition (i), $f \notin NI(G)^2$.
    \item For any $i \in C$, condition (ii) implies that there exists an edge $e \in E(H)$ such that $C \cap e = \{i\}$. Hence, $x_i f \in NI(G)^2$, so $x_i \in NI(G)^2 : f$.
    \item For any $j \in V(G) \setminus C$, condition (iii) implies the existence of $j_1, j_2 \in N_G(j)$ such that $N_G[j_1], N_G[j_2] \subseteq V(G) \setminus C$, and $\dist_G(j_1, j_2) = 2$ while $\dist_{G-j}(j_1, j_2) \geq 3$. Hence, $x_j f$ is divisible by $x_{N_G[j_1]} x_{N_G[j_2]}$. Thus, $x_j \in NI(G)^2 : f$.
\end{itemize}

Therefore, $NI(G)^2 : f = \mm$, completing the proof.
\end{proof}

\begin{rem} Theorem \ref{thm_diam_3} can also be deduced from \cite[Theorem 5.2]{Al}.     
\end{rem}

Based on the result of Hien, Lam, and Trung \cite{HLT}, the maximal ideal $\mm$ is an associated prime of $I(G)^2$ if and only if the graph $G$ contains a dominating triangle. This implies that $\diam(G) \leq 3$. Similarly, we show that $\mm$ is an associated prime of $S/NI(G)^2$, $G$ must contain both a particular cycle structure and a special dominating set, which implies that $\diam(G) \le 6$.

Assume that $\mm$ is an associated prime of $S/NI(G)^2$. By Lemma \ref{lem_quotient}, there exists a squarefree monomial $f$ such that $NI(G)^2 : f = \mm$. Let $U = \supp f$ and $W = \{ i \in U \mid N_G[i] \subseteq U\}$.

\begin{lem}\label{lem_directed_cycle}
With the notation above, we have:
\begin{enumerate}
    \item $W \neq \emptyset$;
    \item $\dist_G(i,j) \le 2$ for all $i,j \in W$;
    \item the induced subgraph of $G$ on $W$ contains a cycle.
\end{enumerate}
\end{lem}

\begin{proof}
(i) Since $U = \supp (f)$ where $\mm = NI(G)^2 : f$ for some squarefree monomial $f$, we deduce that $U$ is non-empty. By Theorems~\ref{thm_diam_2} and~\ref{thm_diam_3}, for any $u \in U$, there exist $v_1, v_2 \in N_G(u)$ such that $N_G[v_1]$ and $N_G[v_2]$ are subsets of $U$. By the definition of $W$, this means $v_1, v_2 \in W$. Thus, $W \neq \emptyset$.

\vspace{1em} 

(ii) Assume for the sake of contradiction that $\dist_G(i,j) \ge 3$ for some $i,j \in W$. This implies that their closed neighborhoods are disjoint, i.e., $N_G[i] \cap N_G[j] = \emptyset$. Since $i,j \in W$, their closed neighborhoods $N_G[i]$ and $N_G[j]$ are both subsets of $U$. Therefore, $N_G[i] \cup N_G[j] \subseteq U$. This leads to a contradiction, as $C = V(G) \setminus U$ does not cover the edge $N_G[i] \cup N_G[j]$ of $H$.

\vspace{1em}

(iii) To establish this, we construct a directed graph $D$ with vertex set $W$. By Theorems~\ref{thm_diam_2} and~\ref{thm_diam_3}, for each $u \in W \subseteq U$, there exist $v_1, v_2 \in W$ such that $\dist_G(v_1, v_2) = 2$ and $\dist_{G - u}(v_1, v_2) \geq 3$. For each $u$, we fix one such pair $(v_1, v_2)$ and add directed edges $(u, v_1)$ and $(u, v_2)$ to $D$. Note that $v_1, v_2 \in N_G(u)$, so the underlying undirected graph of $D$ is a subgraph of the induced subgraph of $G$ on $W$.

By construction, every vertex in $D$ has out-degree exactly 2. Since $W$ is finite, following directed edges from any starting vertex eventually leads to a repeated vertex, which implies the existence of a directed cycle in $D$.

If any such directed cycle has length at least $3$, we are done. Suppose instead that all directed cycles in $D$ have length $2$. Then we may remove such a cycle and repeat the argument on the remaining vertices. If at any stage we find a longer cycle, we are done. Otherwise, we conclude that $D$ is a disjoint union of directed $2$-cycles. In this case, the underlying undirected graph of $D$ is a graph in which every vertex has degree exactly 2, implying that it is a disjoint union of (undirected) cycles. Therefore, the induced subgraph of $G$ on $W$ contains a cycle, completing the proof.
\end{proof}

\begin{cor}\label{cor_diam_bound}
Let $G$ be a simple graph and let $NI(G)$ be its closed neighborhood ideal. Assume that $\mm$ is an associated prime of $NI(G)^2$. Then $\diam(G) \leq 6$.    
\end{cor}

\begin{proof}
We may assume that $G$ is connected and that $\diam(G) \geq 3$. We use the notation from Lemma~\ref{lem_directed_cycle}. By Theorem~\ref{thm_diam_3}, every vertex $v \in C$ has a neighbor in $U$. Moreover, such a neighbor in $U$ is dominated by some vertex in $W$. By Lemma~\ref{lem_directed_cycle}, any two vertices in $W$ are at distance at most $2$ in $G$. Therefore, for any pair of vertices in $G$, the maximum possible distance is at most $6$, completing the proof.
\end{proof}

\begin{cor}
Let $G = C_n$ be the cycle on $n$ vertices. Then $\mm$ is an associated prime of $NI(G)^2$ if and only if $n = 5$.     
\end{cor}

\begin{proof}
First, assume that $n \leq 5$. Then $\diam(G) \le 2$. It is straightforward to verify that $G$ is vertex diameter-$2$-critical if and only if $n = 5$. Hence, by Theorem~\ref{thm_diam_2}, $\mm$ is an associated prime of $S/NI(G)^2$ if and only if $n = 5$.

Now, assume that $n > 5$. Then $\diam(G) \geq 3$. Suppose, for contradiction, that $\mm$ is an associated prime of $S/NI(G)^2$. Then by Lemma~\ref{lem_quotient}, there exists a squarefree monomial $f$ such that $NI(G)^2 : f = \mm$, and let $U = \supp(f)$ be a proper subset of $V(G)$. By Lemma~\ref{lem_directed_cycle}, the induced subgraph of $G$ on $U$ must contain a cycle, a contradiction.
\end{proof}

Note that cycles are special cases of circulant graphs. We now consider the class of cubic circulant graphs. Recall that for a positive integer $n \geq 3$ and a subset $S \subseteq \{1, \ldots, \lfloor n/2 \rfloor\}$, the \emph{circulant graph} generated by $S$, denoted by $C_n(S)$, is the graph with vertex set $[n]$, where two distinct vertices $i$ and $j$ are adjacent if and only if either $|i - j| \in S$ or $n - |i - j| \in S$.

\begin{prop}
Let $G$ be a cubic circulant graph. Then $\mm$ is an associated prime of $S/NI(G)^2$ if and only if $G \cong C_8(1,4)$.
\end{prop}

\begin{proof}
First, by \cite[Theorem 1.1]{NV}, we may assume that $G$ is connected. By a result of Davis and Domke \cite{DD}, any connected cubic circulant graph is isomorphic to either $C_{2n}(1,n)$ for some $n \geq 2$, or $C_{2n}(2,n)$ for some odd integer $n \geq 3$. We first consider the case $G = C_{2n}(1,n)$.

Assume $n \leq 4$. Then $\diam(G) \leq 2$, and one can verify that $G$ is vertex diameter-$2$-critical if and only if $n = 4$. Thus, by Theorem~\ref{thm_diam_2}, $\mm$ is an associated prime of $S/NI(G)^2$ if and only if $n = 4$, i.e., $G \cong C_8(1,4)$.

Now assume $n > 4$, so $\diam(G) \geq 3$. Suppose, for contradiction, that $\mm$ is an associated prime of $S/NI(G)^2$. Using the notation from Lemma~\ref{lem_directed_cycle}, the induced subgraph of $G$ on $W$ must contain a cycle, and for any $u, v \in W$, we have $\dist_G(u, v) = 2$. By the structure of cubic circulant graphs (see \cite[Lemma 2.5]{HPV}), we deduce that $W$ must consist of the vertices of a $4$-cycle. In particular, $|W| = 4$. Since $N_G[W] \subseteq U$, so $|U| \geq 8$. Moreover, for each $u \in U$, there exists a pair $(v_1, v_2) \subseteq U$ such that $N_G(v_1) \cap N_G(v_2) = \{u\}$ and $\dist_{G - u}(v_1, v_2) \geq 3$. Each such pair must be distinct for different $u \in U$, implying at least $8$ distinct pairs. This contradicts the fact that $|W| = 4$.

The case $G = C_{2n}(2,n)$ with $n \geq 3$ odd follows by a similar argument.
\end{proof}

We now present additional examples of graphs for which $\mm$ is an associated prime of $NI(G)^2$, illustrating the richness of this class. The first example demonstrates that the graph $G$ may contain a leaf. The second example exhibits such a graph with diameter $5$. We do not know if such a graph of diameter $6$ exists.

\begin{exm}
Let $G$ be the graph shown below. 
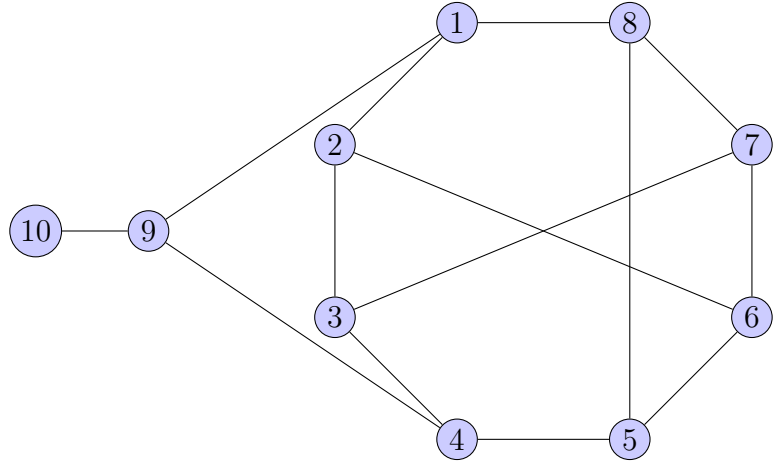
\begin{center}
    \begin{tikzpicture}[scale=1.5,auto, vertex/.style={circle,draw,fill=blue!20,inner sep=2pt}]

\foreach \i/\angle in {1/112.5,2/157.5,3/202.5,4/247.5,5/292.5,6/337.5,7/22.5,8/67.5} {
    \node[vertex] (\i) at (\angle:2cm) {\i};
}

\node[vertex] (9) at (-3.5,0) {9};
\node[vertex] (10) at (-4.5,0) {10};

\draw (1)--(2);
\draw (1)--(8);
\draw (1)--(9);
\draw (2)--(3);
\draw (2)--(6);
\draw (3)--(4);
\draw (3)--(7);
\draw (4)--(5);
\draw (4)--(9);
\draw (5)--(6);
\draw (5)--(8);
\draw (6)--(7);
\draw (7)--(8);
\draw (9)--(10);

\end{tikzpicture}
\end{center}
The closed neighborhood ideal of $G$ is 
$$(x_1x_2x_8x_9,x_1x_2x_3x_6,x_2x_3x_4x_7,x_3x_4x_5x_9,x_4x_5x_6x_8,x_2x_5x_6x_7,x_3x_6x_7x_8,x_1x_5x_7x_8,x_9x_{10}).$$
We have 
$$NI(G)^2 : (x_1\cdots x_9) = (x_1,\ldots,x_{10}).$$    
\end{exm}

\begin{exm}
Let $G$ be the graph whose edge ideal is
\begin{align*}
    I(G) = (&x_1x_2,x_1x_{10},x_2x_3,x_2x_{12},x_2x_{15},x_3x_4,x_3x_{11},x_3x_{13},x_3x_{17},x_4x_5,x_4x_{11},\\
    &x_4x_{12},x_4x_{13},x_4x_{16},x_5x_6,x_5x_{13},x_5x_{18},x_6x_7,x_7x_8,x_7x_{14},x_7x_{17},x_8x_9,\\
    &x_8x_{12},x_8x_{16},x_8x_{17},x_8x_{18},x_9x_{10},x_9x_{13},x_9x_{16},x_9x_{17},x_{10}x_{11},x_{11}x_{12},\\
    &x_{11}x_{14},x_{12}x_{18},x_{13}x_{14},x_{14}x_{15},x_{15}x_{16},x_{15}x_{18},x_{17}x_{18}).
\end{align*}
Then $G$ has $\dist_G(1,6) = 5$ and its neighborhood ideal is 
\begin{align*}
    NI(G) = (&x_1x_2x_{10},x_5x_6x_7,x_1x_9x_{10}x_{11},x_2x_{14}x_{15}x_{16}x_{18},x_4x_{5}x_6x_{13}x_{18},x_6x_7x_8x_{14}x_{17},\\
    &x_4x_8x_{9}x_{15}x_{16},x_7x_{11}x_{13}x_{14}x_{15},x_1x_2x_3x_{12}x_{15},x_5x_8x_{12}x_{15}x_{17}x_{18},\\
    &x_3x_{7}x_8x_9x_{17}x_{18},x_2x_{4}x_8x_{11}x_{12}x_{18},x_8x_9x_{10}x_{13}x_{16}x_{17},x_{2}x_{3}x_4x_{11}x_{13}x_{17},\\
    &x_{3}x_{4}x_{5}x_{9}x_{13}x_{14},x_{3}x_{4}x_{10}x_{11}x_{12}x_{14},x_{7}x_{8}x_{9}x_{12}x_{16}x_{17}x_{18},x_3x_4x_5x_{11}x_{12}x_{13}x_{16}).
\end{align*}
We have $NI(G)^2 : (x_2\cdots x_5 x_7 \cdots x_{18}) = (x_1,\ldots,x_{18})$. Hence, $\mm$ is an associated prime of $NI(G)^2$.
\end{exm}

Finally, we discuss Kneser graphs for which $\mm$ is an associated prime of the second power of their closed neighborhood ideal.

\begin{defn}
Let $k < n$ be positive integers. The \emph{Kneser graph} $K(n,k)$ is the graph whose vertices correspond to the $k$-element subsets of an $n$-element set, where two vertices are adjacent if and only if the corresponding subsets are disjoint.
\end{defn}

\begin{prop}\label{prop_Kneser}
Let $n > k$ be positive integers, and let $K(n,k)$ denote the Kneser graph of $k$-element subsets of an $n$-element set. Assume that $n \geq 3k - 1$. Then $\mm$ is an associated prime of $S/NI(K(n,k))^2$ if and only if $k \geq 2$ and $n = 3k - 1$.
\end{prop}

\begin{proof}
For simplicity, let $G = K(n,k)$. By \cite[Theorem 1]{VV}, we have $\diam(G) = 2$ when $n \geq 3k - 1$. By Theorem~\ref{thm_diam_2}, it suffices to prove that $G$ is vertex diameter-$2$-critical if and only if $n = 3k - 1$.

First, assume $n = 3k - 1$. Let $u = \{1, \ldots, k\}$, and define two other vertices: $v = \{k+1, \ldots, 2k\}$ and $w = \{2k, \ldots, 3k - 1\}$. Note that $|v \cap w| = 1$, so $v$ and $w$ are not adjacent in $G$. Furthermore, $[n] \setminus (v\cap w) = u$. Hence, $\dist_{G - u}(v, w) \geq 3$. This shows that $G$ is vertex diameter-2-critical.

Now assume $n \geq 3k$. Let $u = \{1, \ldots, k\}$. For any pair of non-adjacent vertices $v, w \in V(G) \setminus \{u\}$, we have $|v \cup w| \leq 2k - 1$, so
\[
|[n] \setminus (v \cup w)| \geq n - (2k - 1) \geq k + 1.
\]
This implies that there exists at least one vertex other than $u$ in the intersection $N_G[v] \cap N_G[w]$. Therefore, $v$ and $w$ are connected by a path of length at most $2$ in $G - u$, so $\diam(G - u) = 2$. Hence, $G$ is not vertex diameter-$2$-critical.
\end{proof}

Note that if $n \leq 2k$, then $G$ is disconnected. Thus, the remaining critical range to consider is $k \geq 3$ and $2k + 1 \leq n \leq 3k - 2$. We make the following conjecture:

\begin{conj}
Let $G = K(n,k)$ be the Kneser graph. Then $\mm$ is an associated prime of $NI(G)^2$ if and only if $k \geq 2$ and $n = 3k - 1$.
\end{conj}

\end{document}